\documentclass[a4paper,12pt]{article}
\usepackage[latin1]{inputenc}
\usepackage{amsthm}
\usepackage{amsmath,amssymb,amsfonts}
\usepackage[mathscr]{euscript}
\usepackage{enumerate}

\usepackage[pdftex]{hyperref}
\usepackage{url}
\usepackage[multiple]{footmisc}

\numberwithin{equation}{section}
\newtheorem{thm}{Theorem}[section]

\newtheorem{prop}[thm]{Proposition}
\newtheorem{lm}[thm]{Lemma}
\newtheorem{cor}[thm]{Corollary}
\newtheorem{Rq}[thm]{Remark}

\theoremstyle{definition}

\theoremstyle{remark}

\theoremstyle{plain}

\DeclareMathAlphabet{\calptmx}{OMS}{ztmcm}{m}{n}

\newcommand{\R}{\mathbb{R}}
\newcommand{\Sn}{\mathbb{S}^n}

\newcommand{\Cc}{\mathcal{C}}

\newcommand{\Sb}{\mathbb{S}}
\newcommand{\non}{\noindent}

\newcommand{\bound}{(1+\|W\|^{\frac{1}{2}}_{\mathcal{C}^1}+\|V\|_{\Cc^1})}
\newcommand{\boundinf}{(1+\|W\|^{1/2}_{\infty}+\|V\|_\infty)}

\author{Bakri Laurent
\footnote{Laboratoire de Math\'ematiques et Physique Th\'eorique (UMR CNRS 6083),
 Universit\'e Fran\c{c}ois Rabelais, Parc de Grandmont, 37200 Tours, France.}
\footnote{ E-mail:  laurent.bakri@gmail.com } 
\and 
Casteras Jean-Baptiste 
\footnote{Laboratoire de Math\'ematiques (UMR CNRS 6205), Universit\'e de Bretagne Occidentale, 6 av. Victor Le Gorgeu, 29238 Brest Cedex 3,
France}
\footnote{ E-mail:  jean-baptiste.casteras@univ-brest.fr}
}

\title{ Quantitative uniqueness for Schr\"odinger operator 
with regular 
 potentials}
\begin{document}
\date{}
\maketitle
\medskip

 \begin{abstract}
 We give a sharp upper bound on the vanishing order of solutions to Schr\"odinger equation with $\Cc^1$ electric and magnetic potentials on a compact smooth manifold.
Our method is based on quantitative Carleman type inequalities developed by  Donnelly and Fefferman. It also extends  the first author's previous work to the magnetic potential case.
 \end{abstract}



\section{Introduction}
Let $(M,g)$ be a smooth,  compact, connected, $n$-dimensional Riemannian manifold.
The aim of this paper is to obtain quantitative estimate on the vanishing order of  solutions to  
\begin{equation}\label{E}\Delta u+V\cdot \nabla u+ Wu=0.\end{equation}
We are concerned with $H^1$, non-trivial, solutions to \eqref{E} and $\Cc^1$ potentials (\emph{i.e}  $W$ is a $\Cc^
1$-function on $M$ and $V$ is a  $\Cc^1$-vector field). 
 Recall that the vanishing order at a point $x_0\in M$ of a $L^2$-function $u$ is
\[\inf\left\{d>0\ ;\ \  \limsup_{r\rightarrow 0}\frac{1}{r^{d}}\left(\frac{1}{r^n}\int_{B_r(x_0)}|u(x)|^2dv_g(x)\right)^\frac{1}{2}>0\right\}. \] 
With this setting our main result is  the following 
\begin{thm}\label{van}
The vanishing order of solutions to \eqref{E} is everywhere less than \[C\bound,\]
where $C$ is a positive constant depending only on $(M,g)$.
\end{thm}

Let us first discuss briefly our result. We recall that a differential operator $P$ satisfies the strong unique continuation property (SUCP) if the vanishing order of any non-trivial solutions to $Pu=0$ is finite everywhere. 
There has been an extensive literature dealing with (SUCP)  for solutions to \eqref{E} with singular potentials. 
One of the most useful method to establish (SUCP) is based on Carleman type estimates, some of the principal contributions 
 to \eqref{E} can be found in (\cite{Aron,JK,kenig, KT, Regs,Sogge, Wolff}). We  particularly refer to \cite{KT} which is, to our knowledge, the strongest results up to now.\newline 
  As can be seen in Theorem \ref{van}, our goal is to derive a quantitative version of this unique continuation property.
 Let us now briefly recall some of the principal results already known in this field.  
\\   In the particular case of eigenfunctions of the Laplacian  ($W=\lambda$ and $V=0$), it is a celebrated result of   Donnelly and  Fefferman \cite{DF1}
that the vanishing order is bounded by $C\sqrt{\lambda}$. 
In view of this, it seems a natural conjecture (cf \cite{kenig,K})  that for solutions to $\Delta u + W u=0$, the vanishing order 
is uniformly  bounded by \[C(1+\|W\|_\infty^{\frac{1}{2}}).\]
However, this conjecture is not true when one allows complex valued potentials and solutions. In this complex case, 
it is known that the optimal exponent on $\|W\|_\infty$ is $\frac{2}{3}$ (see \cite{B2,kenig}).      
  When $W$ is a real bounded function and $V=0$,  Kukavica established in $\cite{K}$ some quantitative results for solutions to \eqref{E}. His method is  based on the frequency function (see also \cite{L}) 
which was introduced by  Garofalo and  Lin in \cite{GL} as an alternative to Carleman estimate for (SUCP).   
He established that the vanishing order of solutions is everywhere less than : 
  \[C(1+\sqrt{\|W\|_{\infty}}+\left(\mathrm{osc}\:(W)\right)^2),\] where $\mathrm{osc}(W)=\sup W-\inf W$ and $C$ a constant depending only on $(M,g)$.
  \\
  If $W$ is
  $\Cc^1$, the  first author established  in \cite{B1} the upper bound  \[C(1+\|W\|^{\frac{1}{2}}_{\Cc^1}),\] with $\|W\|_{\Cc^1}=\|W\|_{\infty}+\|\nabla W\|_{\infty}$ and where the exponent $\frac{1}{2}$ is sharp.
 In the general case of equation \eqref{E}, it seems that the first algebraic upper bound, depending on $\|V\|_\infty$ and $\|W\|_\infty$, is given in \cite{B2} 
where it is shown that it is everywhere less than 
\[C(1+\|V\|^2_\infty +\|W\|^{\frac{2}{3}}_\infty).\]

For the real case with magnetic potential, I. Kukavica conjectured  in  \cite{K}    that the vanishing order of solutions is less than \[C(1+\|V\|_\infty +\|W\|^{\frac{1}{2}}_\infty).\]
Finally  in \cite{Linakawangib} (see also \cite{linakawangduke}) quantitative uniqueness is shown for singular potentials. This means that vanishing order is everywhere bounded by a constant, which is no longer explicit. 
Our method  is based on $L^2$-Carleman estimate (Theorem \ref{tics}) in the same spirit as \cite{DF1} : establish a Carleman estimate on the involved operator 
(here : $P:u\mapsto \Delta u+V\cdot \nabla u+Wu$) which is only true for great parameter $\tau$, and state explicitly how $\tau$ depends on the $\Cc^1$ norms of the potentials $V,$ $W$.   

\noindent Our Carleman estimate will allow us to derive the following doubling inequality \begin{equation}\label{doubl}\|u\|_{L^2(B_{2r}(x_0))}\leq e^{C\bound}\|u\|_{L^2(B_r(x_0))}.\end{equation}
This doubling estimate implies Theorem \ref{van}.

The paper is organized as follows. In section 2 we establish Carleman estimates for the operator $P:u\mapsto\Delta u+V\cdot\nabla u+W u$. 
Our method involves repeated  integration by parts in the radial and spherical variables. For the  sake of clarity, a part of the computation is sent to the appendix. 

In section 3, we deduce, in a standard manner, a three balls property for solutions to \eqref{E}. Then using that $M$ is compact we derive a doubling 
inequality which gives immediately Theorem \ref{van}. Finally, in section 4, we show the sharpness of our result with respect to the power of the norms of the potentials $V$ and $W$. 
\subsection{Notations.}

For a fixed point $x_0$  in $M$ we will use the following standard notations:
\begin{itemize}
\item $\Gamma_1(TM)$ will denote the set of $\Cc^1$ vector fields on $M$.
\item $r:=r(x)=d(x,x_0)$ stands for the Riemannian distance from $x_0$,\\
\item $B_r:=B_r(x_0)$ denotes the geodesic ball centered at  $x_0$ of radius $r$,\\
\item $A_{r_1,r_2}:=B_{r_2}\setminus B_{r_1}$.\\
\item $\varepsilon$ stands for a fixed number with $0<\varepsilon<1$. 
\item $R_0,R_1,c,C,C_1,C_2$ will denote positive constants which depend only on $(M,g)$. They may change from a line to another. 
\item $\|\cdot\|$ stands for the  $L^2$ norm on $M$ and $\|\cdot\|_A$ the $L^2$ norm on the (measurable) set $A$. In case $T$ is a vector field (or a tensor), 
$\left\|T\right\|$ has to be understood as $\|\ |T|_g\|$. \\
\end{itemize}

\section{Carleman estimates} 

Recall that Carleman estimates are weighted integral inequalities with a weight function $e^{\tau\phi}$, where the function 
$\phi$ satisfies some convexity properties. 
Let us now define the weight function we will use.\newline  
For a fixed number $\varepsilon$ such that  $0<\varepsilon<1$ and $T_0<0$, we define the function $f$ on  $]-\infty,T_0[ $ by $f(t)=t-e^{\varepsilon t}$. 
One can check easily that,  for $|T_0|$ great enough, the function $f$ verifies the following properties:
\begin{equation}\label{f}
\begin{split}
& 1-\varepsilon e^{\varepsilon T_0}\leq f^\prime(t)\leq 1\quad \forall t\in]-\infty,T_0[,\\
&\displaystyle{\lim _{t\rightarrow-\infty}-e^{-t}f^{\prime\prime}(t)}=+\infty. 
\end{split}
\end{equation}

Finally we define $\phi(x)=-f(\ln r(x))$. Now we can state the main result of this section:

\begin{thm}\label{tics}
 There exist positive constants $R_0, C,C_1$, which depend only on  $M$ and $\varepsilon$, such that, 
for any $\:W\in \Cc^1(M)$, any $V\in\Gamma_1(TM)$, any  $x_0\in M$, any  $u\in C^\infty_0(B_{R_0}(x_0)\setminus\{x_0\})$ and 
 any $\tau \geq C_1(1+\|W\|^{\frac{1}{2}}_{\mathcal{C}^1}+\|V\|_{\Cc^1})$, one has    
\begin{equation}\label{S}
C\left\|r^2e^{\tau\phi}\left(\Delta u +V\cdot \nabla u+Wu \right)\right\|\geq \tau^\frac{3}{2}\left\|r^{\frac{\varepsilon}{2}}e^{\tau\phi}u\right\|
+ \tau^{\frac{1}{2}}\left\|r^{1+\frac{\varepsilon}{2}}e^{\tau\phi}\nabla u\right\|.
\end{equation}
\end{thm}
\vspace{0,4cm}
Under the additional assumption that $\mathrm{supp(u)}$ is far enough from $x_0$ we have the following 

\begin{cor}
\label{cics}
 Adding to the setting of Theorem \ref{tics} the supplementary assumption that 
 \[\mathrm{supp}(u)\subset\{x\in M; r(x)\geq\delta>0\},\] 
 then we have
\begin{equation}\label{Siv2}
\begin{split}
C\left\|r^2e^{\tau\phi}\left(\Delta u+V\cdot \nabla u +Wu \right)\right\|
&\geq
\tau^\frac{3}{2}\left\|r^{\frac{\varepsilon}{2}}e^{\tau\phi}u\right\| \\
+\ \tau^{\frac{1}{2}}\delta^{\frac{1}{2}}\left\|r^{-\frac{1}{2}}e^{\tau\phi}u\right\|
&+
 \tau^{\frac{1}{2}}\left\|r^{1+\frac{\varepsilon}{2}}e^{\tau\phi}\nabla u\right\|.
\end{split}
\end{equation}
\end{cor}
\begin{Rq} 
 One should note that, contrary to the corresponding result of Donnelly and Fefferman \cite{DF1}, we are not able to claim that the constants 
appearing above  depend only on an upper bound of the absolute value of the sectional curvature. This comes from the fact that,
 working in polar coordinates, we will have to handle terms containing spherical derivative of the  metric during the proof of Theorem \ref{tics}.
 See in particular the computation of $I_3$ in the appendix.\end{Rq}

\begin{Rq}We will proceed to the proof with the assumption that all functions are real valued.
 However it can be easily seen that the same inequality holds with hermitian product for complex valued functions.\end{Rq}
\begin{proof}[{\bf Proof of Theorem \ref{tics}}]
We now introduce the polar geodesic coordinates $(r,\theta)$ near $x_0$.
 Using Einstein notation, the Laplace operator takes the form 
\[r^2\Delta u=r^2\partial_r^2u+r^2\left(\partial_r\ln(\sqrt{\gamma})+\frac{n-1}{r}\right)\partial_ru+
\frac{1}{\sqrt{\gamma}}\partial_i(\sqrt{\gamma}\gamma^{ij}\partial_ju),\]
where $\displaystyle{\partial_i=\frac{\partial}{\partial\theta_i}}$ and for each fixed $r$,  $\ \gamma_{ij}(r,\theta)$  
is a metric on \: $\Sb^{n-1}$, and we write $\displaystyle{\gamma=\mathrm{det}(\gamma_{ij})}$.
Since $(M,g)$ is smooth, we have for $r$ small enough :  
\begin{equation}\label{m1}
 \begin{split}
\partial_r(\gamma^{ij})&\leq C (\gamma^{ij})\ \ \ \mbox{(in the sense of tensors)}; 
 \\
|\partial_r(\gamma)|&\leq C;\\
C^{-1}\leq\gamma&\leq C.
\end{split}
\end{equation}

\non Now  we set  $r=e^t$. In these new variables, we  write : 
\begin{align*}
e^{2t}\Delta u&=  \partial_t^2u+(n-2+\partial_t\mathrm{ln}\sqrt{\gamma})\partial_tu
+\frac{1}{\sqrt{\gamma}}\partial_i(\sqrt{\gamma}\gamma^{ij}\partial_ju),\\
e^{2t}V&=  e^{2t}V_t\partial_t+e^{2t}V_i\partial_i .
\end{align*}
Notice that we will consider the function $u$ to have support in  $]-\infty,T_0[\times\mathbb{S}^{n-1},$ 
where $|T_0|$ will be chosen large enough.
The conditions (\ref{m1}) become
\begin{equation}\label{m2}
\begin{split}
\partial_t(\gamma^{ij})&\leq Ce^t (\gamma^{ij})\nonumber\  \ \ \mbox{(in the sense of tensors)};  \\
|\partial_t(\gamma)|&\leq Ce^t;\\
C^{-1}\leq\gamma &\leq  C. \nonumber
\end{split}
\end{equation}
Now we introduce the conjugate operator : 
\begin{equation}
L_\tau(u)=e^{2t}e^{\tau\phi}\Delta(e^{-\tau\phi}u)+e^{2t}e^{\tau\phi}g(V,\nabla(e^{-\tau\phi}u))+e^{2t}Wu,
\end{equation}
and we compute $L_\tau(u)$ : 
\begin{align*}
L_\tau(u)=\ & {} \partial^2_tu+\left(2\tau f^\prime+e^{2t}V_t+n-2+\partial_t\mathrm{ln}\sqrt{\gamma}\right)\partial_tu+e^{2t}V_i\partial_iu\\
               +\  & {}\left(\tau^2f^{\prime^2}+\tau f^{\prime}V_te^{2t}+\tau f^{\prime\prime}+(n-2)\tau f^{\prime}+\tau\partial_t\mathrm{ln}\sqrt{\gamma}f^{\prime}\right)u\\
               +\ & {} \Delta_\theta u+e^{2t}Wu,
\end{align*}
with \[\Delta_\theta u=\frac{1}{\sqrt{\gamma}}\partial_i\left(\sqrt{\gamma}\gamma^{ij}\partial_ju\right).\]

\noindent It will be useful for us to introduce the following $L^2$ norm on $]-\infty,T_0[\times\Sb^{n-1} $: 
\[\|v\|_f^2=\int_{]-\infty,T_0[\times\Sb^{n-1}} |v|^2\sqrt{\gamma}{f^{\prime}}^{-3}dtd\theta,\] where $d\theta$ is the usual measure on $\Sb^{n-1}$.
The corresponding inner product is denoted by  $\left\langle\cdot,\cdot\right\rangle_f$ , \emph{i.e} \[\langle u,v\rangle_f = \int uv\sqrt{\gamma}{f^{\prime}}^{-
3}dtd\theta.\] 
\noindent We will estimate from below $\|L_\tau u\|^2_f$ by using elementary algebra and integrations by parts. 
We are concerned, in the computation, by the power of $\tau$ 
and  exponential decay when $t$ goes to $-\infty$. We point out that we 
have \[e^{t}(V_t+V_i+\partial_\alpha V_t +\partial_\beta V_i)\leq C\left\|V\right\|_{\Cc^1}\] with 
the convention that $\partial_\alpha ,\partial_\beta =\left\{\partial_t ,\partial_1,\ldots ,\partial_{n-1} \right\}$. First note that by triangular inequality one has
\begin{equation}
\|L_\tau(u)\|^2_f\geq \frac{1}{2}I-I\!I ,
\end{equation}
with 
\begin{multline}
I=\left\|\partial^2_tu+\Delta_\theta u+(2\tau f^\prime+e^{2t}V_t)\partial_tu+e^{2t}V_i\partial_iu\phantom{f^{\prime^3}}\!\!\!\right.\\
\left.+\ (\tau^2f^{\prime^2}+\tau f^\prime e^{2t}V_t+ (n-2)\tau f^\prime+e^{2t}W)u\right\|^2_f,
\end{multline}   
and 
\begin{equation}\label{II}
I\!I=\left\|\tau f^{\prime\prime}u+\tau\partial_t\mathrm{ln}\sqrt{\gamma}f^{\prime}u+
(n-2)\partial_tu+\partial_t\ln\sqrt{\gamma}\partial_tu\right\|^2_f.\ \ \ \ \ \ \ 
\end{equation}
%

\non We will be able to absorb $I\!I$ later. 
Now, we want to find a lower bound for $I$. Therefore, we start by computing it :
\[I=I_1+I_2+I_3,\] with 
\begin{equation}
\begin{split}
I_1&=\left\|\partial^2_tu+(\tau^2f^{\prime^2}+\tau f^{\prime}e^{2t}V_t+(n-2)\tau f^\prime +e^{2t}W)u+\Delta_\theta u\right\|_f^2,\\
I_2&=\left\|(2\tau f^\prime+e^{2t}V_t)\partial_tu+e^{2t}V_i\partial_iu\right\|_f^2, \\
I_3&=2 \left\langle \partial^2_tu+(\tau^2f^{\prime^2}+\tau f^{\prime}e^{2t}V_t+(n-2)\tau f^\prime +e^{2t}W)u+\Delta_\theta u \right.\\
& \left.\ \ \ \ \ \ \ \ \ \ \ \ \ \ \ \ \ \ \ \ \ \ \ \ \ \ \ \ \ \ \ \ \ \ \ \ ,(2\tau f^\prime+e^{2t}V_t)\partial_tu+e^{2t}V_i\partial_iu \right\rangle_f .
\end{split}
\end{equation}
We will split the computation into three parts corresponding to the $I_i$ for $i=1,2,3$.
\newline
\newline

\noindent {\bf Computation of $I_1$.}\\
 Let $\rho>0$ be a small number to be chosen later. Since  $|f^{\prime\prime}|\leq1$ and $\tau\geq1$, we have : \newline
 \begin{equation}
 \label{I_1bis}
 I_1\geq\frac{\rho}{\tau}I_1^\prime,\end{equation}
 where $I_1^\prime$ is defined by :
 \begin{equation}I_1^\prime=\left\|\sqrt{|f^{\prime\prime}|}\left[\partial_t^2 u+(\tau^2f'^2+\tau f' e^{2t}V_t+(n-2)\tau f^\prime +e^{2t}W)u+\Delta_\theta u\right]\right\|_f^2 .\end{equation}
Now, we decompose $I_1^\prime$ into three parts
\begin{equation}\label{I1prime}I_1^\prime=K_1+K_2+K_3,\end{equation} 
with 
\begin{align}
 \label{K1} K_1=& {} \left\|\sqrt{|f^{\prime\prime}|}\left(\partial_t^2u+\Delta_\theta u\right)\right\|_f^2, \\
\label{K2} K_2=& {} \left\|\sqrt{|f^{\prime\prime}|}\left(\tau^2f^{\prime^2}+\tau f' e^{2t}V_t+(n-2)\tau f^{\prime}+e^{2t}W\right)u\right\|_f^2,  \\
\label{K3}K_3=&{}2\left\langle\left(\partial_t^2u+\Delta_\theta u\right)\left|f^{\prime\prime}\right|,\left(\tau^2f^{\prime^2}+\tau f' e^{2t}V_t+(n-2)\tau f^{\prime}+e^{2t}W\right)u\right\rangle_f.
\end{align}
We just ignore $K_1$ since it is positive. To estimate $K_2$, we first note that 
\begin{equation*}K_2 \geq \dfrac{\tau^4}{2} \left\|\sqrt{|f^{\prime\prime}|}f^{\prime^2}u\right\|_f^2- \left\|\sqrt{|f^{\prime\prime}|} \left(\tau f' e^{2t}V_t+(n-2)\tau f^{\prime}+e^{2t}W\right)u\right\|_f^2.\end{equation*}
On the other hand, using that $\tau\geq C(1+\|V\|_{\Cc^1}+\|W\|_{\Cc^1}^{\frac{1}{2}})$ and that $f^\prime $ is close to $1$, we have
\begin{multline*}\left\|\sqrt{|f^{\prime\prime}|}\left(\tau f' e^{2t}V_t+(n-2)\tau f^{\prime}+e^{2t}W\right)u\right\|_f^2  \\ \leq c\tau^4\left\|\sqrt{|f^{\prime\prime}|}e^tu\right\|_f^2+ 
\tau^2\left\|\sqrt{|f^{\prime\prime}|}u\right\|^2_f .\end{multline*}
Therefore  using the assumptions on $\tau$, and the exponential decay at $-\infty$, we have for $T_0$ large enough,
 that every other term in $K_2$ can be absorbed in $\tau^4\|\sqrt{|f^{\prime\prime}|}u\|$. 
That is : 
\begin{equation}
\label{K_2}
K_2\geq c\tau^4\int|f^{\prime\prime}||u|^2f^{\prime^{-3}}\sqrt{\gamma}dtd\theta .
\end{equation}  
Now, we derive a suitable lower bound for $K_3$. Integrating by parts gives : 
\begin{equation}\label{K_3}\begin{split}
K_3&=2\int f^{\prime\prime}\left(\tau^2f^{\prime^2}+\tau f' e^{2t}V_t+(n-2)\tau f^{\prime}+e^{2t}W\right)|\partial_tu|^2f^{\prime^{-3}}\sqrt{\gamma}dtd\theta\\
&+2\int \partial_t\left[f^{\prime\prime}\left(\tau^2f^{\prime^2}+\tau f' e^{2t}V_t+(n-2)\tau f^{\prime}+e^{2t}W \right)\right]\partial_tuu\sqrt{\gamma}f^{\prime^{-3}}dtd\theta \\
&-6\int \left(f^{\prime\prime^2}f^{\prime^{-1}}\left(\tau^2f^{\prime^2}+\tau f' e^{2t}V_t+(n-2)\tau f^{\prime}+e^{2t}W \right)\right)\partial_tuu\sqrt{\gamma}f^{\prime^{-3}}dtd\theta \\
&+2\int f^{\prime\prime}\left(\tau^2f^{\prime^2}+\tau f' e^{2t}V_t+(n-2)\tau f^{\prime}+e^{2t}W\right)\partial_t\mathrm{ln}\sqrt{\gamma}\partial_tuuf^{\prime^{-3}}\sqrt{\gamma}dtd\theta\\
&+2\int f^{\prime\prime}\left(\tau^2f^{\prime^2}+\tau f' e^{2t}V_t+(n-2)\tau f^{\prime}+e^{2t}W\right)|D_\theta u|^2f^{\prime^{-3}}\sqrt{\gamma}dtd\theta\\
&+2\int f^{\prime\prime}e^{2t}\partial_iW\cdot\gamma^{ij}\partial_juuf^{\prime^{-3}}\sqrt{\gamma}dtd\theta \\
&+ 2\tau \int  f^{\prime\prime}f^\prime e^{2t}\partial_i V_t\cdot\gamma^{ij}\partial_juuf^{\prime^{-3}}\sqrt{\gamma}dtd\theta,
\end{split}\end{equation}
where $|D_\theta u|^2$ stands for
\[|D_\theta u|^2=\partial_iu\gamma^{ij}\partial_ju .\]
 The condition $\tau\geq C(1+\|V\|_{\Cc^1}+\|W\|_{\Cc^1}^{\frac{1}{2}})$, Young's inequality and the fact that $f^\prime$ is close to $1$ imply 
\begin{multline*}\int |f^{\prime\prime}|e^{2t}|(\partial_iW +f^\prime \partial_i V_t)\gamma^{ij}\partial_juu|f^{\prime^{-3}}\sqrt{\gamma}dtd\theta \\
\leq c\tau^2\int |f^{\prime\prime}|(|D_\theta u|^2+|u|^2)f^{\prime^{-3}}\sqrt{\gamma}dtd\theta.\end{multline*}
Now since  $2\partial_tuu\leq u^2+|\partial_tu|^2$, we can  use conditions \eqref{f} and \eqref{m2} to get

\begin{equation}
\label{K_3bis}
 K_3\geq-c\tau^2\int|f^{\prime\prime}|\left(|\partial_tu|^2+|D_\theta u|^2+|u|^2\right)f^{\prime^{-3}}\sqrt{\gamma}dtd\theta .\\
\end{equation}
Therefore, inserting \eqref{K_2}, \eqref{K_3bis} in \eqref{I1prime} (recall that $K_1\geq 0$), we have
\begin{multline}I_1^\prime \geq -c\tau^2\int|f^{\prime\prime}|\left(|\partial_tu|^2+|D_\theta u|^2+|u|^2\right)f^{\prime^{-3}}\sqrt{\gamma}dtd\theta \\ +c\tau^4\int|f^{\prime\prime}||u|^2f^{\prime^{-3}}\sqrt{\gamma}dtd\theta. \end{multline}
From the definition of $I_1^\prime$ (see \eqref{I_1bis}), we get 
\begin{multline}\label{I_1}
I_1\geq-\rho c\tau\int|f^{\prime\prime}|\left(|\partial_t u|^2+|D_\theta u|^2\right)f^{\prime^{-3}}\sqrt{\gamma}dtd\theta\\
 +C\tau^{3}\rho\int|f^{\prime\prime}||u|^2f^{\prime^{-3}}\sqrt{\gamma}dtd\theta.
\end{multline}
\newline
{\bf Computation of $I_2$.}\\
We begin by recalling that 
\[I_2=\left\|(2\tau f^\prime+e^{2t}V_t)\partial_tu+e^{2t}V_i\partial_iu\right\|_f^2.\] 
\noindent 
In the same way as for $I_1$, using that $\tau \geq 1$, we have
\[I_2 \geq \dfrac{1}{\tau}I_2.\]
Using the triangular inequality, one has
\[I_2 \geq \dfrac{1}{2}\left\|2\tau f^\prime \partial_t u\right\|_f^2-\left\|e^{2t}V_t\partial_tu+e^{2t}V_i\partial_iu\right\|_f^2.\]
Now, using the assumptions on $\tau$, we note that \[\left\|e^{2t}V_t\partial_tu+e^{2t}V_i\partial_iu\right\|_f^2\leq 2\tau^2\|e^{t}\partial_tu\|_f^2+2\tau^2\|e^{t}D_\theta u|\|^2_f .\]
From the last three previous inequalities and since $e^t$ is small for $T_0$ largely negative, 
we see that the following estimate holds
\begin{equation}
\label{I_2}
 I_2\geq c\tau \|\partial_t u\|_f^2 -c\tau \|e^t D_\theta u\|^2_f .
 \end{equation}

\noindent {\bf Computation of $I_3$.}\\
Since this computation is quite lengthy, we send it to the Appendix. There, we show that
 \begin{equation}\label{I_3}\begin{split}
 I_3 &\geq  3\tau \int\left|f^{\prime\prime}\right|\left|D_\theta u\right|^2f^{\prime^{-3}}\sqrt{\gamma}dtd\theta -c\tau^3\int e^t|u|^2f^{\prime^{-3}}\sqrt{\gamma}dtd\theta\\
 &\quad-c\tau\int\left|f^{\prime\prime}\right|\left|\partial_t u\right|^2f^{\prime^{-3}}\sqrt{\gamma}dtd\theta
-c\tau^2\int\left|f^{\prime\prime}\right| |u|^2f^{\prime^{-3}}\sqrt{\gamma}dtd\theta
.
 \end{split}\end{equation}

\noindent {\bf Lower bound for $L_\tau u$.}\\
 
Now recalling that $I=I_1+I_2+I_3$ and using \eqref{I_1}, \eqref{I_2} and \eqref{I_3}, we obtain
\begin{equation}
\label{pro}
\begin{split}
I&\geq  c \tau\int|\partial_t u|^2 f^{\prime^{-3}}\sqrt{\gamma}dtd\theta +3\tau \int\left|f^{\prime\prime}\right|\left|D_\theta u\right|^2f^{\prime^{-3}}\sqrt{\gamma}dtd\theta \nonumber\\
& + C\tau^3\rho\int|f^{\prime\prime}||u|^2 f^{\prime^{-3}}\sqrt{\gamma}dtd\theta-c\rho\tau\int|f^{\prime\prime}||D_\theta u|^2f^{\prime^{-3}}\sqrt{\gamma}dtd\theta \nonumber\\
&- c\tau\int e^{2t}|D_\theta u|^2 f^{\prime^{-3}}\sqrt{\gamma}dtd\theta-c\tau^3\int e^t|u|^2f^{\prime^{-3}}\sqrt{\gamma}dtd\theta \nonumber\\
&-c\tau\int|f^{\prime\prime}| |\partial_t u|^2 f^{\prime^{-3}}\sqrt{\gamma}dtd\theta- c\tau^2\int|f^{\prime\prime}||u|^2 f^{\prime^{-3}}\sqrt{\gamma}dtd\theta.
\end{split}
\end{equation}

\noindent Now we want to derive a lower bound for $I$. 
Then one needs to check that every non-positive term in the right hand side of \eqref{pro} can be absorbed. 
\\We  first fix $\rho$ small enough (\emph{i.e.} $\rho\leq \frac{2}{c}$)  such that
\[\rho c\tau\int|f^{\prime\prime}|\cdot|D_\theta u|^2{f^{\prime}}^{-3}\sqrt{\gamma}dtd\theta\leq 2\tau\int|f^{\prime\prime}|\cdot|D_\theta u|^2f^{\prime^{-3}}\sqrt{\gamma}dtd\theta\]
where $c$ is the constant appearing in \eqref{pro}. Now the other negative terms 
of \eqref{pro} can then be absorbed by comparing powers of $\tau$ and decay rate at $-\infty$. 
Indeed conditions \eqref{f} imply that $e^t$ is small compared to $|f^{\prime\prime}|$.\newline
%
%
%
\noindent Thus we obtain :
\begin{equation}\label{ssu}\begin{split}
 C I&\geq \tau\int|\partial_t u|^2f^{\prime^{-3}}\sqrt{\gamma}dtd\theta+\tau\int|f^{\prime\prime}||D_\theta u|^2f^{\prime^{-3}}\sqrt{\gamma}dtd\theta\\
&\quad + \tau^{3}\int|f^{\prime\prime}||u|^2f^{\prime^{-3}}\sqrt{\gamma}dtd\theta .
\end{split}\end{equation}
Now  we can check that $I\!I$ can be absorbed in $I$ for  $|T_0|$ and $\tau$ large enough. 
Indeed from \eqref{II}, using \eqref{f} and \eqref{m2} ($|\partial_t \mathrm{ln}\sqrt{\gamma}|\leq C e^t$), one gets 
\begin{equation}\begin{split}
I\!I&= \left\|\tau f^{\prime\prime}u+\tau\partial_t\mathrm{ln}\sqrt{\gamma}f^{\prime}u+(n-2)\partial_tu+\partial_t\ln\sqrt{\gamma}\partial_tu\right\|^2_f \\
&\leq \tau^2\int|f^{\prime\prime}|^2|u|^2 f^{\prime^{-3}}\sqrt{\gamma}dtd\theta+\tau^2\int e^{2t}|u|^2f^{\prime^{-3}}\sqrt{\gamma}dtd\theta \\
&\quad+ C\int|\partial_t u|^2 f^{\prime^{-3}}\sqrt{\gamma}dtd\theta.
\end{split}\end{equation}
And each term in the right hand side can easily be absorbed in \eqref{ssu}. 
Then we obtain 
\begin{equation}\label{bla}\|L_\tau u\|_f^2\geq C\tau^3\|\sqrt{|f^{\prime\prime}|} u\|_f^2+C\tau\|\partial_t u\|_f^2
+C\tau\|\sqrt{|f^{\prime\prime}|}D_\theta u\|_f^2 .\end{equation}
Note that, since $\sqrt{|f^{\prime\prime}|}\leq1$, one has
 \begin{equation}
 \label{premiermars1}
 \|L_\tau u\|_f^2\geq C\tau^3\|\sqrt{|f^{\prime\prime}|} u\|_f^2+c\tau\|\sqrt{|f^{\prime\prime}|}\partial_t u\|_f^2
+C\tau\|\sqrt{|f^{\prime\prime}|}D_\theta u\|_f^2,\end{equation}
and the constant $c$ can be chosen arbitrarily smaller than $C$. \\

\noindent {\bf End of the proof.}\\
If we set  $v=e^{-\tau\phi}u$ and use the triangular inequality on the second right-sided term of \eqref{premiermars1}, then we have 
\begin{multline}\left\|e^{2t}e^{\tau\phi}(\Delta v+V\cdot\nabla v+Wv)\right\|_f^2\geq C\tau^3\left\|\sqrt{|f^{\prime\prime}|}e^{\tau\phi} v\right\|_f^2
-c\tau^3\left\|\sqrt{|f^{\prime\prime}|}f^\prime e^{\tau\phi} v\right\|_f^2\\
+
\frac{c}{2}\tau\left\|\sqrt{|f^{\prime\prime}|}e^{\tau\phi}\partial_t v\right\|_f^2+C\tau\left\|\sqrt{|f^{\prime\prime}|}e^{\tau\phi} D_\theta v\right\|_f^2\end{multline}
Finally since $f^\prime$ is close to 1 one can absorb the negative term to obtain
 
\begin{multline}\left\|e^{2t}e^{\tau\phi}(\Delta v+V\cdot\nabla v+Wv)\right\|_f^2\geq C\tau^3\left\|\sqrt{|f^{\prime\prime}|}e^{\tau\phi} v\right\|_f^2\\
+
C\tau\left\|\sqrt{|f^{\prime\prime}|}e^{\tau\phi}\partial_t v\right\|_f^2+C\tau\left\|\sqrt{|f^{\prime\prime}|}e^{\tau\phi} D_\theta v\right|_f^2\end{multline}
It remains to get back to the usual $L^2$ norm. First note that since $f^\prime$ is close to 1,
 we can get the same estimate without the term $(f^\prime)^{-3}$ in the integrals. 
Recall that in polar coordinates $(r,\theta)$ the volume element 
is $r^{n-1}\sqrt{\gamma}drd\theta$, we can deduce from \eqref{ssu} that : 
\begin{multline}
  \|r^2e^{\tau\phi}(\Delta v+V\cdot\nabla v+Wv)r^{-\frac{n}{2}}\|^2\geq C\tau^3\|r^\frac{\varepsilon}{2}e^{\tau\phi}vr^{-\frac{n}{2}}\|^2\\
+C\tau\|r^{1+\frac{\varepsilon}{2}}e^{\tau\phi}\nabla vr^{-\frac{n}{2}}\|^2.
 \end{multline}
Finally one can get rid of the term $r^{-\frac{n}{2}}$ by replacing $\tau$ with $\tau+\frac{n}{2}$. Indeed, from 
$e^{\tau\phi}r^{-\frac{n}{2}}=e^{(\tau+\frac{n}{2})\phi}e^{-\frac{n}{2}r^\varepsilon}$, one can easily check that, for $r$ small enough  
\[\frac{1}{2}e^{(\tau+\frac{n}{2})\phi}\leq e^{\tau\phi}r^{-\frac{n}{2}}\leq e^{(\tau+\frac{n}{2})\phi}.\] 

\non This achieves the  proof of Theorem \ref{tics}.\\ 
\end{proof}
Next, we demonstrate Corollary \ref{cics}.
\begin{proof}[{\bf Proof of Corollary \ref{cics}.}]
\non Now suppose that  $\mathrm{supp}(u)\subset\{x\in M; r(x)\geq\delta>0\}$ and define $T_1=\ln\delta$.\newline

\noindent Cauchy-Schwarz inequality applied to \[\int\partial_t(u^2)e^{-t}\sqrt{\gamma}dtd\theta=2\int u\partial_tue^{-t}\sqrt{\gamma}dtd\theta\] 
gives
\begin{equation}\label{d1}
 \int\partial_t(u^2)e^{-t}\sqrt{\gamma}dtd\theta\leq 2\left(\int\left(\partial_tu\right)^2e^{-t}\sqrt{\gamma}dtd\theta \right)^{\frac{1}{2}}
\left(\int u^2e^{-t}\sqrt{\gamma}dtd\theta\right)^{\frac{1}{2}}.
\end{equation}
On the other hand, integrating by parts gives
 \begin{equation}
      \int\partial_t(u^2)e^{-t}\sqrt{\gamma}dtd\theta = \int u^2e^{-t}\sqrt{\gamma}dtd\theta-\int u^2e^{-t}\partial_t(\ln(\sqrt{\gamma}))\sqrt{\gamma}dtd\theta.
     \end{equation}
Now since  $|\partial_t\ln\sqrt{\gamma}|\leq Ce^t$ for $|T_0|$ large enough we can deduce :
\begin{equation}\label{d2}
 \int\partial_t(u^2)e^{-t}\sqrt{\gamma}dtd\theta \geq c  \int u^2e^{-t}\sqrt{\gamma}dtd\theta.
\end{equation}
Combining \eqref{d1}, \eqref{d2} and by the assumption on $\mathrm{supp}(u)$, we find  
\begin{equation*}\label{prout}\begin{split}
 c^2 \int u^2e^{-t}\sqrt{\gamma}dtd\theta&\leq 4\int\left(\partial_tu\right)^2e^{-t}\sqrt{\gamma}dtd\theta\\
&\leq4e^{-T_1}\int\left(\partial_t u\right)^2\sqrt{\gamma}dtd\theta.\end{split}\end{equation*}

\noindent Finally, dropping all terms except  $\tau \int|\partial_t u|^2f^{\prime^{-3}}\sqrt{\gamma}dtd\theta$  in \eqref{ssu}  gives :

\begin{eqnarray*}
CI\geq \tau \delta \int e^{-t}|u|^2{f^\prime}^{-3}\sqrt{\gamma}dtd\theta.
\end{eqnarray*}\vspace{0,5cm}
Inequality \eqref{ssu} can then be replaced by :  
\begin{equation}\begin{split}
 I &\geq C\tau\int|\partial_t u|^2f^{\prime^{-3}}\sqrt{\gamma}dtd\theta+C\tau\int|f^{\prime\prime}|\cdot|D_\theta u|^2f^{\prime^{-3}}\sqrt{\gamma}dtd\theta\\
&+ C\tau^{3}\int|f^{\prime\prime}|\cdot|u|^2f^{\prime^{-3}}\sqrt{\gamma}dtd\theta+C\tau \delta\int e^{-t}|u|^2{f^\prime}^{-3}\sqrt{\gamma}dtd\theta.
\end{split}\end{equation}
The rest of the proof follows in a way similar to the last part of the proof of Theorem \ref{tics}. 
\end{proof}
\section{Vanishing order}

We now proceed to establish an upper bound on the vanishing order of solutions to \eqref{E}, from our Carleman estimate. This is inspired by \cite{DF1}. 
We choose to establish a doubling inequality. We recall that doubling inequality implies vanishing order estimate. 
Before proceeding, we would like to emphasize that if  $u\in H^1(B_r(x_0))$, 
by standard elliptic regularity theory, one has that $u\in H^2_\mathrm{loc}(B_r(x_0))$ (see for example \cite{GT} Theorem 8.8). 
Therefore, by density,  we see that we can apply inequality \eqref{Siv2} of Corollary \ref{cics} to $\chi u$ for $\chi$ a cut-off function null 
in a neighborhood of $x_0$
\subsection{Three balls inequality}
We first want to derive from \eqref{Siv2}, a control on the local behavior of solutions in the form of an Hadamard
 three circles type theorem. To obtain such result the basic idea is to apply Carleman estimate to $\chi u$ where 
$\chi$ is an appropriate cut-off function and $u$  a solution of \eqref{E}.  This is standard \cite{B1,JL} and the proof  
adapted  to our weight function  is given for the sake of completeness. 

\begin{prop}[Three balls inequality]\label{tst}
There exist positive constants $R_1$, $C_1$, $C_2$ and $0<\alpha<1$  which depend only on $(M,g)$ such that, if $u$ is a solution to \eqref{E} with $W\in \Cc^1 (M) $ and $V\in \Gamma_1(TM)$,  
then for any $R<R_1$, and any $x_0\in M, $ one has 
\begin{equation}\label{ttc}
\| u\|_{B_R(x_0)}\leq e^ {C\bound}\|u\|_{B_{\frac{R}{2}}(x_0)}^\alpha\| u\|_{B_{2R}(x_0)}^{1-\alpha}.
\end{equation}
\end{prop}

\begin{proof}
Let $x_0$ be a point in $M$. Let $u$ be a solution 
to \eqref{E} and $R$ such that $0<R<\frac{R_0}{2}$ with $R_0$ as in Corollary \ref{tics}. 
Let $\psi\in\Cc^{\infty}_0(B_{2R})$, $0\leq\psi\leq1$, a function with the following properties:
\begin{itemize}
\item $\psi(x)=0$ if $r(x)<\frac{R}{4}$ or  $r(x)>\frac{5R}{3}$,
\item[$\bullet$] $\psi(x)=1$ if $\frac{R}{3}<r(x)<\frac{3R}{2}$,
\item[$\bullet$] $|\nabla\psi(x)|\leq \frac{C}{R}$,
 \item[$\bullet$] $|\nabla^{2}\psi(x) |\leq \frac{C}{R^2}$.
\end{itemize}
First, since the function $\psi u$ is supported in the annulus $A_{\frac{R}{3},\frac{5R}{3}}$, we can apply estimate \eqref{Siv2} of Theorem \ref{tics}. 
In particular, since the quotient between $\frac{R}{3}$ and $\frac{5R}{3}$ doesn't depend on $R$, we have
\begin{equation}\label{debutm}C\left\|r^2e^{\tau\phi}\left(\Delta \psi u+2\nabla u\cdot\nabla\psi+V\cdot u\nabla\psi\right)\right\|\geq\tau^{\frac{1}{2}}\left\|e^{\tau\phi}\psi u\right\|.\end{equation}
Notice that 
\[\|r^2e^{\tau\phi}V\cdot u\nabla\psi\|\leq \|V\|_\infty\|r^2e^{\tau\phi}u\nabla\psi\|. \] 
Then, from the properties of $\psi$ and  since $\tau\geq\|V\|_\infty$,  we get 
\begin{align}
\tau^{\frac{1}{2}}\|e^{\tau\phi}\psi u\|
\leq &
\ C\left(\|e^{\tau\phi}u\|_{\frac{R}{4},\frac{R}{3}}+\|e^{\tau\phi}u\|_{\frac{3R}{2},\frac{5R}{3}}\right)\label{u}\\
 &\quad +C\left(R\|e^{\tau\phi}\nabla u\|_{\frac{R}{4},\frac{R}{3}}+R\|e^{\tau\phi}\nabla u\|_{\frac{3R}{2},\frac{5R}{3}}\right)\\
 &\quad +\ C\tau\|re^{\tau \phi} u\|_{\frac{R}{4},\frac{R}{3}}+C\tau\|re^{\tau \phi} u\|_{\frac{3R}{2},\frac{5R}{3}}\label{V}
\end{align}
Now since $r$ is  small, we bound \eqref{V} and the right hand side of \eqref{u}   from above by $\tau\|e^{\tau \phi} u\|_{\frac{R}{4},\frac{R}{3}}+\tau\|e^{\tau \phi} u\|_{\frac{3R}{2},\frac{5R}{3}}$. 
Then, dividing both sides of the previous inequality by $\tau$ and noticing that $\tau^{-\frac{1}{2}}\leq 1$, one has :
\begin{multline}
\|e^{\tau\phi}u\|_{\frac{R}{3},\frac{3R}{2}} \leq  C\tau^\frac{1}{2}\left(\|e^{\tau\phi}u\|_{\frac{R}{4},\frac{R}{3}}+\|e^{\tau\phi}u\|_{\frac{3R}{2},\frac{5R}{3}}\right) \\
+C\left(R\|e^{\tau\phi}\nabla u\|_{\frac{R}{4},\frac{R}{3}}+R\|e^{\tau\phi}\nabla u\|_{\frac{3R}{2},\frac{5R}{3}}\right).\end{multline}

\non Recall that  $\phi(x)=-\ln r(x)+r(x)^\varepsilon$. In particular $\phi$ is radial and decreasing (for small $r$). Then one has, 
\begin{multline}\label{3c7}
\|e^{\tau\phi}u\|_{\frac{R}{3},\frac{3R}{2}}\leq
C\tau^\frac{1}{2}\left(e^{\tau\phi(\frac{R}{4})}\|u\|_{\frac{R}{4},\frac{R}{3}}+e^{\tau\phi(\frac{3R}{2})}\|u\|_{\frac{3R}{2},\frac{5R}{3}}\right)
\\+C\left(Re^{\tau\phi(\frac{R}{4})}\|\nabla u\|_{\frac{R}{4},\frac{R}{3}}
+Re^{\tau\phi(\frac{3R}{2})}\|\nabla u\|_{\frac{3R}{2},\frac{5R}{3}}\right).
\end{multline}
Now we recall the following elliptic estimates : since $u$ satisfies \eqref{E} then  :
 \begin{equation}\label{lm1m}\|\nabla u\|_{aR}\leq C\left(\frac{1}{(1-a)R}+\|W\|^{1/2}_{\infty}+\|V\|_\infty\right)\|u\|_{R}, \:\:\ \mathrm{for}\  \:0<a<1.   \end{equation}
Moreover since $A_{R_1,R_2}\subset B_{R_2}$ and $\tau \geq 1$, multiplying formula \eqref{lm1m} by $e^{\tau\phi(\frac{3R}{2})}$, we find

\[e^{\tau\phi(\frac{3R}{2})}\| \nabla u\|_{\frac{3R}{2},\frac{5R}{3}}\leq C\tau^\frac{1}{2}\left(\frac{1}{R}+\|W\|^{1/2}_{\infty}+\|V\|_\infty\right)e^{\tau\phi(\frac{3R}{2})}\|u\|_{2R}.\]
Using  (\ref{3c7}) and noting that $\|e^{\tau\phi} u\|_{\frac{R}{3},\frac{3R}{2}}\geq e^{\tau\phi(R)}\|u\|_{\frac{R}{3},{R}}$, one has :
 \begin{equation*}\label{ind}
\|u\|_{\frac{R}{3},R} \leq C\tau^{\frac{1}{2}}(1+\|W\|^{1/2}_{\infty}+\|V\|_\infty)\left( e^{\tau A_R}\|u\|_{\frac{R}{2}}+e^{-\tau B_R}\|u\|_{2R}\right),
  \end{equation*}
  with $A_R=\phi(\frac{R}{4})-\phi(R)$ and  $B_R=-(\phi(\frac{3R}{2})-\phi(R))$.
  From the properties of  $\phi$  we may assume that we have $0<A^{-1}\leq A_R\leq A$ and $0<B\leq B_R\leq B^{-1}$ where $A$ and $B$ don't depend on $R$. 
We may assume that $C\tau^\frac{1}{2}(1+\|W\|^{1/2}_{\infty}+\|V\|_\infty)\geq 2$. Then we can add $\|u\|_{\frac{R}{3}}$ to each side and bound it in the right hand side by  
$C\tau^\frac{1}{2}\boundinf e^{\tau A}\|u\|_{\frac{R}{2}}$. We get :
  
\begin{equation}\label{3c2}
\|u\|_{R} \leq  C\tau^\frac{1}{2}\boundinf\left(e^{\tau A}\|u\|_{\frac{R}{2}}+e^{-\tau B}\|u\|_{2R}\right).
  \end{equation} 
Now we want to find $\tau$ such that \[C\tau^\frac{1}{2}\boundinf e^{-\tau B}\|u\|_{2R}\leq \frac{1}{2}\|u\|_{R}\]
which is true for $\tau \geq -\frac{2}{B}\ln\left(\frac{1}{2C\boundinf}\frac{\|u\|_R}{\|u\|_{2R}}\right).$ 
Since $\tau$ must also satisfy  \[\tau \geq C_1\bound,\]
we choose
\begin{equation*}
\tau = -\frac{2}{B}\ln\left(\frac{1}{2C\boundinf}\frac{\|u\|_R}{\|u\|_{2R}}\right)+C_1\bound.
\end{equation*}

\noindent Since, of course, $ \|U\|_{\Cc^1}\geq \|U\|_\infty$, one has : 
\begin{equation}\label{last}
\|u\|_R^{\frac{B+2(A+1)}{B}}\leq e^{C\bound}\|u\|_{2R}^{\frac{2(A+1)}{B}}\|u\|_{\frac{R}{2}},
\end{equation}
 
\noindent  Finally, defining  $\alpha=\frac{2(A+1)}{2(A+1)+B}$, we see that \eqref{last} gives the result.

\end{proof}

\subsection{Doubling estimates}
Now we intend to show that the vanishing order of solutions to $\eqref{E}$ is everywhere bounded by $C\bound$. 
This is an immediate consequence of the following :
\begin{thm}[doubling estimate]\label{dor}

There exists a positive constant  $C$,  depending only on  $(M,g)$ such that : if $u$ is a solution to   
\eqref{E} on $M$ 
then  for any $x_0$ in $M$ and any  $r>0$, one has
\begin{equation}\label{do}\|u\|_{B_{2r}(x_0)}\leq e^{C\bound}\|u\|_{B_r(x_0)}.
\end{equation}
\end{thm}

\noindent  To prove Theorem \ref{dor}, we need to use the standard overlapping chains of balls argument (\cite{DF1,JL,K}) to show :  
 \begin{prop}\label{cor1}
For any $R>0$ there exists $C_R>0$ such that for any  $x_0\in M$, any $W\in \Cc^1(M)$, any $V\in \Gamma_1(TM)$, and any solutions $u$ to \eqref{E} :
\[\| u\|_{B_R(x_0)}\geq e^{-C_R(1+\|W\|^{\frac{1}{2}}_{\mathcal{C}^1}+\|V\|_{\Cc^1})}\| u\|_{L^2(M)} .\]
\end{prop}

\begin{proof}
We may assume without loss of generality that $R<R_1$, with $R_1$ as in the three balls inequality (Proposition \ref{ttc}).
 Up to multiplication by a constant, we can assume that $\| u\|_{L^2(M)}=1$.
 We denote by $\bar{x}$ a point in $M$ such that  $\| u\|_{B_R(\bar{x})}=\sup_{x\in M}  \|u\|_{B_R(x)}$.
 This implies that one has  $\| u\|_{B_{R(\bar{x})}}\geq D_R$, where $D_R$ depends only on $M$ and $R$. 
One has (from Proposition \ref{ttc})  at an arbitrary point $x$ of $M$ : 

\begin{equation}\label{cop}\| u\|_{B_{R/2}(x)}\geq e^{-c\bound}\| u\|^{\frac{1}{\alpha}}_{B_R(x)}.\end{equation} 

Let $\gamma$ be a geodesic curve  between $x_0$ and $\bar{x}$ and define  $x_1,\cdots,x_m=\bar{x}$ such that 
 $x_i\in\gamma$ and
 $B_{\frac{R}{2}}(x_{i+1})\subset B_R(x_i),$ for any  $i$ from $0$ to $m-1$. The number $m$  depends only on $\mathrm{diam}(M)$ and  $R$.
 Then the properties of $(x_i)_{1\leq i\leq m}$ and inequality \eqref{cop} give for all $i$, $1\leq i\leq m$ :
\begin{equation}
\|u\|_{B_{R/2}(x_i)}\geq e^{-c\bound}\|u\|^{\frac{1}{\alpha}}_{B_{R/2}(x_{i+1})}.
\end{equation}

The result follows by iteration and the fact that $\| u\|_{B_R(\bar{x})}\geq D_R$.

\end{proof}
\begin{cor}\label{cor2}
For all $R>0$, there exists a positive constant $C_R$ depending only on $M$ and $R$ such that at any point $x_0$ in $M$ one has
\begin{equation*}
\| u\|_{R,2R}\geq e^{-C_R\bound}\| u\|_{L^2(M)}.
\end{equation*}
\end{cor}
\begin{proof} 
 Recall that $\|u\|_{R,2R}=\|u\|_{L^2(A_{R,2R})}$ with $A_{R,2R}:=\{x; R\leq d(x,x_0)\leq 2R)\}$.
Let  $R<R_1$ where $R_1$ is from Proposition \ref{tst}, note that $R_1\leq \mathrm{diam}(M)$. 
Since $M$ is geodesically complete, there exists a point $x_1$ in  $A_{R,2R}$  
 such that $B_{x_1}(\frac{R}{4})\subset A_{R,2R}$. From Proposition \ref{cor1}, one has 
 \[\|u\|_{B_{\frac{R}{4}}(x_1)}\geq e^{-C_R\bound}\| u\|_{L^2(M)}\]
 which gives the result. 
\end{proof}
\begin{proof}[Proof of Theorem \ref{dor}]

We proceed as in the proof of three balls inequality (Proposition \ref{tst}) except for the fact that now we want the first ball to become arbitrarily small in front of the others.
 Let  $R=\frac{R_1}{4}$ with $R_1$ as in the three balls inequality, let $\delta$ such that  $0<3\delta<\frac{R}{8}$,
and define a smooth function $\psi$, with $0\leq\psi\leq1$  as follows: \vspace{0,3cm}
 
\begin{itemize}
\item[$\bullet$] $\psi(x)=0$ if $r(x)<\delta$ or if $r(x)>R$,
\item[$\bullet$] $\psi(x)=1$ if  $r(x)\in[\frac{5\delta}{4},\frac{R}{2}]$,
\item[$\bullet$] $|\nabla\psi(x)|\leq\frac{C}{\delta}$ and  $|\nabla^2\psi(x)|\leq\frac{C}{\delta^2}$ if $r(x)\in[\delta,\frac{5\delta}{4}]$ ,
 \item[$\bullet$] $|\nabla\psi(x)|\leq C$ and $|\nabla^2\psi(x)|\leq C$ if $r(x)\in[\frac{R}{2},R]$.\vspace{0,3cm}
\end{itemize}
Keeping appropriate terms in \eqref{Siv2} applied to $\psi u$ gives :  
\begin{multline}
\tau^\frac{3}{2}\|r^{\frac{\varepsilon}{2}}e^{\tau\phi}\psi u\|+ \tau^{\frac{1}{2}}\delta^{\frac{1}{2}}\|r^{-\frac{1}{2}}e^{\tau\phi}\psi u\|\\
\leq
C\left(\|r^2e^{\tau\phi}\nabla u\cdot\nabla\psi\|+\|r^2e^{\tau\phi}\Delta\psi u\|+\|r^2e^{\tau\phi}Vu\nabla \psi\|\right). 
\end{multline}
Using properties of $\psi$ and since $\tau\geq\|V\|_\infty$, one finds

\begin{equation}
\label{premiermars2}
\begin{split}
\tau^\frac{3}{2}\|r^{\frac{\varepsilon}{2}}e^{\tau\phi} u\|_{\frac{R}{8},\frac{R}{4}}+\tau^\frac{1}{2}\|e^{\tau\phi}u\|_{\frac{5\delta}{4},3\delta} & \leq  C (\delta \|e^{\tau\phi}\nabla u\|_{\delta,\frac{5\delta}{4}}+\|e^{\tau\phi}\nabla u\|_{\frac{R}{2},R})\nonumber \\
&\quad+C(\|e^{\tau\phi} u\|_{\delta,\frac{5\delta}{4}}+\|e^{\tau\phi}u\|_{\frac{R}{2},R})\nonumber \\
&\quad+  C \dfrac{\tau}{\delta}\|r^2e^{\tau\phi}u\|_{\delta,\frac{5\delta}{4}}+C\tau \| r^2e^{\tau\phi}u\|_{\frac{R}{2},R}.\nonumber \\
\end{split}
\end{equation}
Now, we bound from above the two last terms of the previous inequality by $C\tau\left(\|e^{\tau\phi} u\|_{\delta,\frac{5\delta}{4}}+\|e^{\tau\phi}u\|_{\frac{R}{2},R}\right)$. Then we divide both sides of \eqref{premiermars2} by $\tau^{\frac{1}{2}}$. Noticing that $\tau \geq 1$, this yields to
\begin{equation}\begin{split}
\|r^{\frac{\varepsilon}{2}}e^{\tau\phi} u\|_{\frac{R}{8},\frac{R}{4}}+\|e^{\tau\phi}u\|_{\frac{5\delta}{4},3\delta} &\leq 
 C \tau^{\frac{1}{2}}\left(\delta \|e^{\tau\phi}\nabla u\|_{\delta,\frac{5\delta}{4}}+\|e^{\tau\phi}\nabla u\|_{\frac{R}{2},R}\right)\\
& + C\tau^{\frac{1}{2}}\left(\|e^{\tau\phi} u\|_{\delta,\frac{5\delta}{4}}+\|e^{\tau\phi}u\|_{\frac{R}{2},R}\right).
\end{split}\end{equation} 

 \noindent From the elliptic estimate (\ref{lm1m}) and the decreasing of $\phi$, we get 
\begin{equation*}\begin{split}
e^{\tau\phi(\frac{R}{4})} \|u\|_{\frac{R}{8},\frac{R}{4}}&+ e^{\tau\phi(3\delta)}\|u\|_{\frac{5\delta}{4},3\delta} \\ 
&\leq  C\tau^\frac{1}{2}\boundinf\left(e^{\tau\phi(\delta)}\|u\|_{\frac{3\delta}{2}}+e^{\tau\phi(\frac{R}{3})}\|u\|_{\frac{5R}{3}}\right).
\end{split}\end{equation*}
\noindent Adding $e^{\tau\phi(3\delta)}\|u\|_{\frac{5\delta}{4}}$ to each side and noting that we can bound it from above 
by $C\tau^\frac{1}{2}\boundinf e^{\tau\phi(\delta)}\|u\|_{\frac{3\delta}{2}}$, we find that 
\begin{equation*}\begin{split}
e^{\tau\phi(\frac{R}{4})} \|u\|_{\frac{R}{8},\frac{R}{4}}&+e^{\tau\phi(3\delta)}\|u\|_{3\delta}\\ &\leq 
C\tau^\frac{1}{2}\boundinf\left(e^{\tau\phi(\delta)}\|u\|_{\frac{3\delta}{2}}+e^{\tau\phi(\frac{R}{3})}\|u\|_{\frac{5R}{3}}\right).
\end{split}\end{equation*}
Now we want to choose $\tau$ such that  
\[C\tau^\frac{1}{2}\boundinf e^{\tau\phi(\frac{R}{3})}\|u\|_{\frac{5R}{3}}\leq \frac{1}{2}e^{\tau\phi(\frac{R}{4})} \|u\|_{\frac{R}{8},\frac{R}{4}}.\]
For the same reasons as before we choose 
  \begin{multline*}\tau=\frac{2}{\phi(\frac{R}{3})-\phi(\frac{R}{4})}\mathrm{ln}\left(\frac{1}{2C\boundinf}\frac{\|u\|_{\frac{R}{8},\frac{R}{4}}}{\|u\|_{\frac{5R}{3}}}\right)\\ +C\bound.
\end{multline*}
 Define $D_R=-\left(\phi(\frac{R}{3})-\phi(\frac{R}{4})\right)^{-1}$; like before, one has $0<E^{-1}\leq D_R\leq E$, with $E$ a fixed real number. 
 Dropping the first term in the left hand side and noting that $0<\phi (\delta )-\phi (3\delta)\leq C$, one has
  \[\|u\|_{3\delta}\leq e^{C\bound}\left(\frac{\|u\|_{\frac{R}{8},\frac{R}{4}}}{\|u\|_{\frac{5R}{3}}}\right)^{-E}\|u\|_{\frac{3\delta}{2}} \]
  Finally, from Corollary \ref{cor2}, we define  $r=\frac{3\delta}{2}$ to have : 
  \[\|u\|_{2r}\leq e^{C\bound}\|u\|_{r}. \]
   Thus, the theorem is proved for all $r\leq\frac{R_1}{16}$.
 Using Proposition \ref{cor1} we have for $r\geq \frac{R_1}{16}$ :
\begin{equation*}\begin{split}
\|u\|_{B_{x_0}(r)}\geq\| u\|_{B_{x_0}(\frac{R_0}{16})}&\geq e^{-C_0\bound}\| u\|_{L^2(M)}\\
&\geq e^{-C_1\bound} \|u\|_{B_{x_0}(2r)}.
  \end{split}\end{equation*}
\end{proof}
Finally Theorem \ref{van} is an easy and  direct consequence of this doubling estimate. 
 \section{Sharpness}
In this short section we intend to show that the estimate we obtain in Theorem \ref{van} is sharp. That is to say, in the uniform upper bound on the vanishing order
\[
 C\bound,
\]
one cannot replace the exponents 1 on $\|V\|_{\Cc^1}$ and $1/2$ on $\|W\|_{\Cc^1}$ by lower ones.\\ 
Indeed, consider the function $f_k(x_1,x_2,\cdots,x_{n+1})=\Re e(x_1+ix_2)^k$ defined in $\R^{n+1}$. We set
 $h_k$ to be the restriction of $f_k$ to $\mathbb{S}^n$, so $(h_k)_k$ is a sequence of spherical harmonics and
 $-\Delta_{\Sn} h_k=k(k+n-1)h_k$. For a smooth, non-constant, function $f$ on $M$, we define 
\[\phi_k=e^{kf}h_k.\] First, notice that $\phi_k$ vanishes at order $k$ at the north pole $(0,0,\cdots,1)$. 
Now it is easy to check that
\[
 \Delta_{\Sn}\phi_k=\langle V_k , \phi_k\rangle +W_k\phi_k,
\]
with
 \begin{equation*}
 \begin{split}
  V_k & = 2k\nabla f,\\
W_k&=k\Delta f-k^2|\nabla f|^2-k(k+n-1).
 \end{split}
\end{equation*}
Then one has  $C^{-1}k\leq\left\|V_k\right\|_{\Cc^1}\leq Ck,$ and $\ C^{-1}k^2\leq\left\|W_k\right\|_{\Cc^1}\leq Ck^2$, for an appropriate constant $C$ depending only on $(f,n)$. 
Therefore the sharpness is established. 
\section{Appendix.}
The aim of this appendix is to prove the claim (\ref{I_3}) we used in the proof of Theorem \ref{tics}. More precisely, we show the following lemma.
\begin{lm}
We have  
 \begin{equation}
 \label{I_3bis }
 \begin{split}
 I_3 &\geq  3\tau \int\left|f^{\prime\prime}\right|\left|D_\theta u\right|^2f^{\prime^{-3}}\sqrt{\gamma}dtd\theta 
-c\tau^3\int e^t|u|^2f^{\prime^{-3}}\sqrt{\gamma}dtd\theta\\
 &\quad-c\tau\int\left|f^{\prime\prime}\right|\left|\partial_t u\right|^2f^{\prime^{-3}}\sqrt{\gamma}dtd\theta
-c\tau^2\int\left|f^{\prime\prime}\right| |u|^2f^{\prime^{-3}}\sqrt{\gamma}dtd\theta. 
 \end{split}\end{equation}
 \end{lm}
 \begin{proof}
 We begin by recalling the definition of $I_3$ :
\begin{multline*}
I_3=2 \left\langle \partial^2_tu+(\tau^2f^{\prime^2}+\tau f^{\prime}e^{2t}V_t+(n-2)\tau f^\prime +e^{2t}W)u+\Delta_\theta u \right.\\
 \left.,(2\tau f^\prime+e^{2t}V_t)\partial_tu+e^{2t}V_i\partial_iu \right\rangle_f .
\end{multline*}
We also recall the following estimates on the weight and the metric :
\begin{equation}\label{fbis}
\begin{split}
 &1-\varepsilon e^{\varepsilon T_0}\leq f^\prime(t) \leq 1 \qquad\forall t\in]-\infty,T_0[,\\
&\displaystyle{\lim _{t\rightarrow-\infty}-e^{-t}f^{\prime\prime}(t)}=+\infty, 
\end{split}
\end{equation}
and, $\forall \ i,j,k\in \left\{1,\ldots n-1\right\}$,
\begin{equation}\label{m2bis}
\begin{split}
\partial_t(\gamma^{ij})&\leq Ce^t (\gamma^{ij}) \ \ \mbox{(in the sense of tensors)};  \\
\partial_k(\gamma^{ij})&\leq C (\gamma^{ij})\ \ \ \ \ \mbox{(in the sense of tensors)};  \\
|\partial_t(\gamma)|&\leq Ce^t;\\
C^{-1}\leq\gamma &\leq  C. 
\end{split}
\end{equation}
\noindent
We will also use the key assumption on $\tau$ : \begin{equation}\label{tau}\tau\geq C_1(1+\sqrt{\|W\|_{\mathcal{C}^1}}+\|V\|_{\Cc^1}).\end{equation}

\non In order to compute $I_3$ we write it in a convenient way: 
\begin{equation}\label{I3}
 I_3=\sum_{i=1}^{16}J_i,
\end{equation}
where the integrals $J_i$ are defined by :
\begin{equation*}
\begin{split}
J_1&=2\tau \int f^\prime \partial_t(|\partial_tu|^2)f^{\prime^{-3}}\sqrt{\gamma}dtd\theta\\
J_2&=4\tau\int f^\prime\partial_tu\partial_i\left(\sqrt{\gamma}\gamma^{ij}\partial_ju\right)f^{\prime^{-3}}dtd\theta\\
J_3&=\int \left(2\tau^3+2(n-2)\tau^2f^{{\prime}^{-1}}+2\tau f^{\prime^{-2}} e^{2t}W\right)\partial_t|u|^2\sqrt{\gamma}dtd\theta\\
J_4&=2 \tau^2\int e^{2t}V_t\partial_t|u|^2f^{\prime^{-1}}\sqrt{\gamma}dtd\theta\\
J_5&=\int e^{2t}V_t\partial_t(|\partial_t u|^2)f^{\prime^{-3}}\sqrt{\gamma}dtd\theta\\
J_6&=\tau^2\int e^{2t}V_t\partial_t|u|^2f^{\prime^{-1}}\sqrt{\gamma}dtd\theta\\
J_7&=\tau\int  e^{4t}V^2_t       \partial_t|u|^2f^{\prime^{-2}}\sqrt{\gamma}dtd\theta\\
\end{split}
\end{equation*}
\begin{equation*}
\begin{split}
J_8&=(n-2)\tau\int     e^{2t}V_t     \partial_t|u|^2f^{\prime^{-2}}\sqrt{\gamma}dtd\theta\\
J_9&=\int        e^{4t}WV_t \partial_t|u|^2f^{\prime^{-3}}\sqrt{\gamma}dtd\theta\\
J_{10}&=2\int e^{2t}V_t\partial_i( \sqrt{\gamma}\gamma^{ij}\partial_ju)\partial_tuf^{\prime^{-3}} dtd\theta\\
J_{11}&=2\int e^{2t}V_i\partial_iu\partial^2_tuf^{\prime^{-3}}\sqrt{\gamma}dtd\theta\\
J_{12}&=\tau^{2}\int e^{2t} V_i \partial_i |u|^2f^{\prime^{-1}}\sqrt{\gamma} dtd\theta\\
J_{13}&=\tau\int  e^{4t} V_iV_t \partial_i |u|^2     f'^{-2} \sqrt{\gamma}dtd\theta\\
J_{14}&=(n-2)\tau\int  e^{2t} V_i \partial_i|u|^2      f'^{-2} \sqrt{\gamma} dtd\theta\\
J_{15}&=\int e^{4t}V_iW\partial_i|u|^2f^{\prime^{-3}}\sqrt{\gamma}dtd\theta\\
J_{16}&=  2\int  e^{2t} V_k \partial_k u \partial_i\left(\sqrt{\gamma}\gamma^{ij}\partial_ju\right)  f'^{-3}  dtd\theta  .
\end{split}
\end{equation*}
Here we noticed that $2\partial_tu\partial_t^2u=\partial_t(|\partial_tu|^2)$ and $2 u\partial_t u = \partial_t |u|^2$.
Before we start the computation, we want to point out that the only positive term of \eqref{I3} comes from $J_2$. Now we will use integration by parts to estimate each $J_i$. 
Note that $f$ is radial.\newline
 
\noindent We begin with $J_1$. We find that  :
\begin{equation*}
J_1
=
\int\left(4\tau f^{\prime\prime}\right)|\partial_tu|^2f^{\prime^{-3}}\sqrt{\gamma}dtd\theta\\
%
-
\int2\tau f^\prime\partial_t\mathrm{ln}\sqrt{\gamma}|\partial_{t}u|^2f^{\prime^{-3}}\sqrt{\gamma}dtd\theta.
\end{equation*}
The conditions \eqref{m2bis} imply that $|\partial_t\ln\sqrt{\gamma}|\leq Ce^t$. Then properties \eqref{fbis} on $f$ give for large $|T_0|$ 
that $|\partial_t\ln\sqrt{\gamma}|$ is small compared to $|f^{\prime\prime}|$. Then one has   
\begin{equation}\label{J_1}J_1\geq -c\tau\int |f^{\prime\prime}|\cdot|\partial_tu|^2f^{\prime^{-3}}\sqrt{\gamma}dtd\theta.\end{equation}

\noindent 
In order to estimate $J_2$ we first integrate by parts with respect to $\partial_i$ : 
\begin{equation*}\begin{array}{rcl}
J_2
   &=&-2\int2\tau f^{\prime}\partial_t\partial_iu\gamma^{ij}\partial_juf^{\prime^{-3}}\sqrt{\gamma}dtd\theta.                  
     \end{array}
\end{equation*}            
Then we integrate by parts with respect to $\partial_t$. We get : 
\begin{multline*}
J_2=-4\tau\int f^{\prime\prime}\gamma^{ij}\partial_iu\partial_juf^{\prime^{-3}}\sqrt{\gamma}dtd\theta\\
\quad+\int2\tau f^{\prime}\partial_t\mathrm{ln}\sqrt{\gamma}\gamma^{ij}\partial_iu\partial_juf^{\prime^{-3}}\sqrt{\gamma}dtd\theta\\
+\int2\tau f^\prime\partial_t(\gamma^{ij})\partial_iu\partial_juf^{\prime^{-3}}\sqrt{\gamma}dtd\theta.
\end{multline*}
Recall that $|D_\theta u|^2$ denotes  $|D_\theta u|^2=\partial_iu\gamma^{ij}\partial_ju$. Now using that $-f^{\prime\prime}$ is non-negative and $\tau$ is large, 
the conditions \eqref{fbis}  and \eqref{m2bis} give for $|T_0|$ large enough:
\begin{equation}\label{J_2}
 J_2\geq \dfrac{7}{2}\tau\int|f^{\prime\prime}|\cdot|D_\theta u|^2f^{\prime^{-3}}\sqrt{\gamma}dtd\theta.
\end{equation}
Similarly computation of $J_3$ gives :
\begin{multline*}\label{J_31}
J_3=-2\int(\tau^3+(n-2)\tau^2f^{\prime^{-1}})\partial_t\mathrm{ln}(\sqrt{\gamma})u^2\sqrt{\gamma}dtd\theta\\
-\int(4f^\prime-4 f^{\prime\prime}+2f^{\prime}\partial_t\ln\sqrt{\gamma} )\tau e^{2t}Wu^2f^{\prime^{-3}}\sqrt{\gamma}dtd\theta\\
+2\int(n-2)\tau^2f^{\prime\prime}f^\prime|u|^2f^{\prime^{-3}}\sqrt{\gamma}dtd\theta\\
-\int2\tau f^\prime e^{2t}\partial_tW|u|^2f^{\prime^{-3}}\sqrt{\gamma}dtd\theta.
\end{multline*}

 From  \eqref{fbis} and \eqref{m2bis} one can see that if $C_1$ and $|T_0|$ are large enough, then\:\newline
\begin{equation}\label{J_3}
J_3\geq-c\tau^3\int e^t|u|^2f^{\prime^{-3}}\sqrt{\gamma}dtd\theta
-c\tau^2\int |f^{\prime\prime}|u|^2f^{\prime^{-3}}\sqrt{\gamma}dtd\theta .
\end{equation}
We now compute the terms involving only radial derivatives, that is to say $J_i$ for $i=4,\ldots ,9$. We have
\begin{equation*}\begin{split}
J_4&=2 \tau^2\int e^{2t}V_t\partial_t|u|^2f^{\prime^{-1}}\sqrt{\gamma}dtd\theta
\\
&= -2\tau^2 \int u^2 e^{2t}\left(2V_t f^{\prime^2}-V_t f'' f' + f'^2 \partial_t V_r+f'^2 V_t \partial_t (\ln \sqrt{\gamma})     \right)  f'^{-3} \sqrt{\gamma} dtd\theta
\end{split}
\end{equation*}
\begin{equation}J_4  \geq-c\tau^3\int e^{t}u^2f'^{-3} \sqrt{\gamma} dtd\theta ,\end{equation}
and
\begin{equation*}\begin{split}
J_5& =  \int e^{2t}V_t \partial_t (|\partial_t u|^2)f'^{-3} \sqrt{\gamma} dtd\theta \\
&= - \int e^{2t}|\partial_t u|^2  \left(2V_t+\partial_t V_t -3V_t f'' f'^{-1}+V_t \partial_t (\ln \sqrt{\gamma})   \right)  f'^{-3} \sqrt{\gamma} dtd\theta \\
\end{split}\end{equation*}
\begin{equation}J_5 \geq  -c \tau\int |f^{\prime \prime}| |\partial_t u|^2 f'^{-3} \sqrt{\gamma} dtd\theta .\end{equation}
In the last inequality, we use that $e^t$ is small compared to $|f''|$. Let's resume our computation. We obtain
\begin{equation*}\begin{split}
J_6 &=\tau^{2} \int e^{2t}V_t \partial_t |u|^2 f'^{-1} \sqrt{\gamma} dtd\theta \\
&= -\tau^2\int e^{2t} u^2 \left(2V_t f^{\prime^{2}}+\partial_t V_t f'^2-f''f' V_t+ V_t f'^2 \partial_t (\ln \sqrt{\gamma})  \right)f'^{-3} \sqrt{\gamma} dtd\theta \\
\end{split}\end{equation*}
\begin{equation}J_6 \geq  -c \tau^3\int  e^t u^2f'^{-3} \sqrt{\gamma} dtd\theta ,\end{equation}

\begin{equation*}\begin{split}
J_{7}&=\tau \int  e^{4t}|V_t|^2 \partial_t |u|^2   f'^{-2} \sqrt{\gamma} dtd\theta 
\\
&= -\tau \int u^2e^{4t} \left(4|V_t|^2f^{\prime}+ 2V_t\partial_t V_t  f' \right)f'^{-3} \sqrt{\gamma} dtd\theta \\
&\quad+\tau \int u^2e^{4t} \left(  2|V_t|^{2} f''-|V_t|^{2}f^\prime\partial_t (\ln \sqrt{\gamma})  \right)f'^{-3} \sqrt{\gamma} dtd\theta 
\end{split}\end{equation*}
\begin{equation}
 J_{7}\geq-c\tau^3 \int e^{t}|u|^2  f'^{-3} \sqrt{\gamma} dtd\theta ,
\end{equation}
and
\begin{equation*}\begin{split}
J_{8}&=(n-2)\tau\int e^{2t}V_t\partial_t|u|^2f^{\prime^{-2}}\sqrt{\gamma}dtd\theta\\
&=-(n-2)\tau\int e^{2t}\left( 2V_tf^{\prime}+\partial_tV_tf^{\prime}-2f^{\prime\prime}V_t+V_t\partial_t\ln\sqrt{\gamma}\right)|u|^2f^{\prime^{-3}}\sqrt{\gamma}dtd\theta \\
\end{split}\end{equation*}
\begin{equation}
 J_{8}  \geq-c\tau^2\int |f''| |u|^2\sqrt{\gamma}f^{\prime^{-3}}dtd\theta,
\end{equation}
where we use once more that $e^t$ is small compared to $|f''|$. Finally, for $J_9$, we get
\begin{equation*}\begin{split}
J_9&=\int e^{4t} V_tW \partial_t |u|^2 f'^{-3} \sqrt{\gamma} dtd\theta \\ 
&= 
- \int u^2 e^{4t}\left(4V_t W+\partial_t V_t W+V_t \partial_t W  \right)f'^{-3} \sqrt{\gamma} dtd\theta \\
&\quad +\int u^2 e^{4t}\left(3f'' f'^{-1}V_t W - V_t W  \partial_t (\ln \sqrt{\gamma})  \right)f'^{-3} \sqrt{\gamma} dtd\theta\\
\end{split}\end{equation*}

\begin{equation}
 J_9\geq-c\tau^3\int e^{t}|u|^2f'^{-3} \sqrt{\gamma} dtd\theta.
\end{equation}

\noindent Now, we deal with the terms involving spherical derivative. 
 We recall that
\begin{equation*}\begin{split}
J_{10}&=2\int e^{2t} V_t \partial_t u  \partial_i(\sqrt{\gamma}\gamma^{ij}\partial_j u)  f'^{-3}dtd\theta . \\
\end{split}\end{equation*}
Integrating by parts in the spherical variables gives 
\begin{equation*}\begin{split}
J_{10}& = -2\int e^{2t} V_t \partial_i \partial_t u \gamma^{ij} \partial_j u  f'^{-3} \sqrt{\gamma} dtd\theta \\
&\quad\quad\quad\quad\quad\quad- 2\int e^{2t}\partial_i V_t \partial_t u \gamma^{ij} \partial_j u   f'^{-3} \sqrt{\gamma} dtd\theta .\\
\end{split}\end{equation*}
Now, we use the identity $\partial_t |D_\theta u|^2 =2\gamma^{ij}\partial_t \partial_i u \partial_j u+\partial_t \gamma^{ij}\partial_i u \partial_j u$ to find 
\begin{equation*}\begin{split}
J_{10}&
=-\int e^{2t}V_t (\partial_t |D_{\theta} u|^2-\partial_t \gamma^{ij}  \partial_i u \partial_j u )  f'^{-3} \sqrt{\gamma} dtd\theta \\
&\quad\quad\quad\quad\quad\quad\quad\quad- 2\int e^{2t}\partial_i V_t \partial_t u \gamma^{ij} \partial_j u   f'^{-3} \sqrt{\gamma} dtd\theta.
\end{split}\end{equation*}
Finally, integrating by parts with respect to the radial variable, 
\begin{equation*}\begin{split}
J_{10}&= \int   e^{2t}|D_{\theta} u|^2 (2V_t+\partial_t V_t-3V_t f'' f'^{-1}+V_t \partial_t  (\ln \sqrt{\gamma})  )     f'^{-3} \sqrt{\gamma} dtd\theta \\
&\quad \quad\quad\quad+\int e^{2t}V_t \partial_t \gamma^{ij}  \partial_i u \partial_j u   f'^{-3} \sqrt{\gamma} dtd\theta \\
&\quad\quad\quad\quad\quad\quad\quad\quad\quad-2 \int e^{2t}\partial_i V_t \partial_t u \gamma^{ij} \partial_j u   f'^{-3} \sqrt{\gamma} dtd\theta, \\
\end{split}\end{equation*}
we obtain
\begin{equation}
J_{10} 
\geq  -c\tau\int |f''| |\partial_t u|^2  f'^{-3} \sqrt{\gamma} dtd\theta -c\tau\int e^{t} |D_\theta u|^2 f'^{-3} \sqrt{\gamma} dtd\theta.
\end{equation}
Integrating by parts the following 
\begin{equation*}\begin{split}
J_{11}&=2\int  e^{2t} V_i \partial_i u  \partial^2_t u  f'^{-3} \sqrt{\gamma} dtd\theta,\\
\end{split}\end{equation*}
gives 
\begin{equation*}\begin{split}
J_{11} &= - 2\int e^{2t}\partial_i u\partial_t u  \left(2V_i+ \partial_t V_i-3f'' f'^{-1}V_i+V_i \partial_t (\ln \sqrt{\gamma})   \right)  f'^{-3} \sqrt{\gamma} dtd\theta\\
&\quad\quad\quad\quad\quad\quad\quad\quad  -2\int e^{2t} V_i \partial_t \partial_i u \partial_t u  f'^{-3} \sqrt{\gamma} dtd\theta.
\end{split}\end{equation*}
Noticing that $2\partial_t \partial_i u \partial_t u=\partial_i|\partial_t u|^2$, we have 

\begin{equation*}\begin{split}
J_{11}&=-2\int e^{2t}\partial_i u\partial_t u  \left(2V_i+ \partial_t V_i-3f'' f'^{-1}V_i+V_i \partial_t (\ln \sqrt{\gamma})   \right)  f'^{-3} \sqrt{\gamma} dtd\theta\\
&\quad\quad\quad\quad\quad\quad\quad\quad 
- \int e^{2t} V_i \partial_i (|\partial_t u|^2) f'^{-3} \sqrt{\gamma} dtd\theta,\end{split}\end{equation*}
then integrating by parts the last integral of the right hand side gives 
\begin{equation*}\begin{split}
J_{11}&= - 2\int e^{2t}\partial_i u\partial_t u  \left(2V_i+ \partial_t V_i-3f'' f'^{-1}V_i+V_i \partial_t (\ln \sqrt{\gamma})   \right)  f'^{-3} \sqrt{\gamma} dtd\theta \\
& \quad\quad\quad\quad\quad\quad\quad\quad + \int e^{2t} |\partial_t u|^2 \left(\partial_i V_i +V_i \partial_i (\ln \sqrt{\gamma})\right)  f'^{-3}\sqrt{\gamma} dtd\theta.
\end{split}\end{equation*}

Therefore we can state that 
\begin{equation} 
J_{11} \geq  -c\tau\int |f''| |\partial_t u|^2  f'^{-3} \sqrt{\gamma} dtd\theta -c\tau\int e^{t} |D_\theta u|^2) f'^{-3} \sqrt{\gamma} dtd\theta.
\end{equation}
From the definition of $J_{12}$
\begin{equation*}\begin{split}
 J_{12}&=\tau^{2}\int  e^{2t} V_i \partial_i| u|^2        f'^2   f'^{-3} \sqrt{\gamma} dtd\theta,\\
\end{split}\end{equation*}
integrating by parts with respect to the spherical variables gives 
\begin{equation*}\begin{split}
  J_{12}&= -\tau^2\int u^2 e^{2t} f'^2\left(\partial_i V_i+V_i \partial_i  (\ln \sqrt{\gamma})  \right)f'^{-3} \sqrt{\gamma} dtd\theta, \\
\end{split}\end{equation*}
therefore we can derive the estimate
\begin{equation}\begin{split}
J_{12}&\geq -c\tau^3 \int |u|^2 e^{t} f'^{-3} \sqrt{\gamma} dtd\theta.
\end{split}\end{equation}
In the same way, we have

\begin{equation*}\begin{split}
J_{13}&=\tau\int  e^{4t} V_iV_t  \partial_i |u|^2    f'^{-2} \sqrt{\gamma} dtd\theta \\
 &= -\tau\int u^2 e^{4t} f' \left( \partial_i V_i V_t+
V_i \partial_i V_t  + V_i V_t  \partial_i (\ln \sqrt{\gamma})\right) f'^{-3} \sqrt{\gamma} dtd\theta \\
\end{split}\end{equation*}
\begin{equation}\begin{split}
J_{13}&\geq-c\tau^3\int|u|^2e^{t}f'^{-3} \sqrt{\gamma} dtd\theta ,
\end{split}\end{equation}
\begin{equation*}\begin{split}
 J_{14}&=(n-2)\tau\int  e^{2t} V_i \partial_i|u|^2      f'^{-2} \sqrt{\gamma} dtd\theta\\
    &=-(n-2)\tau\int e^{2t}(\partial_iV_i+\partial_i\ln\sqrt{\gamma})|u|^2f^{\prime^{-2}}\sqrt{\gamma}dtd\theta\\
\end{split}\end{equation*}
\begin{equation}
J_{14}\geq-c\tau^2\int |f''||u|^2f^{\prime^{-3}}\sqrt{\gamma}dtd\theta ,
\end{equation}
and
\begin{equation*}\begin{split}
 J_{15}&=\int e^{4t}V_iW\partial_i|u|^2f^{\prime^{-3}}\sqrt{\gamma}dtd\theta \\
 &=-\int e^{4t}(\partial_iV_iW +V_i\partial_iW+V_iW\partial_i\ln\sqrt{\gamma})|u|^2f^{\prime^{-3}}\sqrt{\gamma}dtd\theta
\end{split}\end{equation*} 
\begin{equation}
 J_{15}\geq-c\tau^3\int e^{t}|u|^2f^{\prime^{-3}}\sqrt{\gamma}dtd\theta .
\end{equation}
We now turn to $J_{16}$ 
\begin{equation*}\begin{split}
J_{16}&=2\int  e^{2t} V_k \partial_k u \partial_i(\sqrt{\gamma}\gamma^{ij}\partial_j u)  f'^{-3} dtd\theta. \end{split}\end{equation*} 
We first integrate by parts with respect to $\partial_t$
\begin{equation*}\begin{split}
J_{16} &= -2\int e^{2t} (\partial_i V_k \partial_k u+V_k \partial_i \partial_k u )\gamma^{ij} \partial_j u  f'^{-3} \sqrt{\gamma}dtd\theta ,
 \end{split}\end{equation*} and use the identity $ \partial_i \partial_k u \gamma^{ij} \partial_j u=\dfrac{1}{2}(\partial_k |D_\theta u|^2- \partial_k \gamma^{ij} \partial_i u \partial_j u )$ to find
\begin{equation*}\begin{split}
J_{16}&=  -2\int e^{2t} \partial_i V_k \partial_k u \gamma^{ij} \partial_j u  f'^{-3} \sqrt{\gamma} dtd\theta \\
&\quad-\int e^{2t} V_k  (\partial_k |D_\theta u|^2- \partial_k \gamma^{ij} \partial_i u \partial_j u ) f'^{-3} \sqrt{\gamma} dtd\theta .
 \end{split}\end{equation*} 
 Then an integration by parts with respect to $\partial_k$ gives  \begin{equation*}\begin{split}
J_{16}&= -2\int e^{2t} \partial_i V_k \partial_k u \gamma^{ij} \partial_j u  f'^{-3} \sqrt{\gamma} dtd\theta \\
&\quad+\int e^{2t} (\partial_k V_k +V_k \partial_k (\ln \sqrt{\gamma}) ) |D_\theta u|^2 f'^{-3} \sqrt{\gamma} dtd\theta \\
&\quad +\int e^{2t} V_k  \partial_k \gamma^{ij} \partial_i u \partial_j u  f'^{-3} \sqrt{\gamma} dtd\theta .\\
 \end{split}\end{equation*} This yields to :  \begin{equation}\label{J_16}\begin{split}
J_{16} & \geq  -c \tau\int |D_\theta u|^2 e^{t} f'^{-3} \sqrt{\gamma} dtd\theta.
\end{split}\end{equation}

 \noindent Therefore, combining all the previous estimates on the $J_i$ (\emph{i.e.} \eqref{J_1} to \eqref{J_16}) and noticing that 
\[c\tau\int |D_\theta u|^2 e^{t} f'^{-3} \sqrt{\gamma} dtd\theta \leq \dfrac{1}{2}\tau \int\left|f^{\prime\prime}\right|\left|D_\theta u\right|^2f^{\prime^{-3}}\sqrt{\gamma}dtd\theta ,\] 
 we have established that 
 \begin{equation*}
 \begin{split}
 I_3 &\geq  3\tau \int\left|f^{\prime\prime}\right|\left|D_\theta u\right|^2f^{\prime^{-3}}\sqrt{\gamma}dtd\theta -c\tau^3\int e^t|u|^2f^{\prime^{-3}}\sqrt{\gamma}dtd\theta\\
 &-c\tau\int\left|f^{\prime\prime}\right|\left|\partial_t u\right|^2f^{\prime^{-3}}\sqrt{\gamma}dtd\theta
-c\tau^2\int\left|f^{\prime\prime}\right| |u|^2f^{\prime^{-3}}\sqrt{\gamma}dtd\theta.
 \end{split}\end{equation*}
\end{proof}

\bibliographystyle{plain}
\bibliography{biblio}

\begin{thebibliography}{10}

\bibitem{Aron}
Nachman Aronszajn.
\newblock A unique continuation theorem for solutions of elliptic partial
  differential equations or inequalities of second order.
\newblock {\em J. Math. Pures Appl. (9)}, 36:235--249, 1957.

\bibitem{B1}
Laurent Bakri.
\newblock Quantitative uniqueness for {S}chr\"odinger operator.
\newblock \emph{Indiana Univ. Math. J. ,} to appear, available at
  \url{http://www.iumj.indiana.edu/IUMJ/Preprints/4713.pdf}.

\bibitem{B2}
Laurent Bakri.
\newblock Vanishing order of solutions to schrodinger equation.
\newblock \emph{preprint,} available at \url{http://arxiv.org/abs/1111.6530},
  2011.

\bibitem{DF1}
Harold Donnelly and Charles Fefferman.
\newblock Nodal sets of eigenfunctions on {R}iemannian manifolds.
\newblock {\em Invent. Math.}, 93(1):161--183, 1988.

\bibitem{GL}
Nicola Garofalo and Fang-Hua Lin.
\newblock Monotonicity properties of variational integrals, {$A\sb p$} weights
  and unique continuation.
\newblock {\em Indiana Univ. Math. J.}, 35(2):245--268, 1986.

\bibitem{GT}
David Gilbarg and Neil~S. Trudinger.
\newblock {\em Elliptic partial differential equations of second order}.
\newblock Classics in Mathematics. Springer-Verlag, Berlin, 2001.
\newblock Reprint of the 1998 edition.

\bibitem{JK}
David Jerison and Carlos~E. Kenig.
\newblock Unique continuation and absence of positive eigenvalues for
  {S}chr\"odinger operators.
\newblock {\em Ann. of Math. (2)}, 121(3):463--494, 1985.
\newblock With an appendix by E. M. Stein.

\bibitem{JL}
David Jerison and Gilles Lebeau.
\newblock Nodal sets of sums of eigenfunctions.
\newblock In {\em Harmonic analysis and partial differential equations
  ({C}hicago, {IL}, 1996)}, Chicago Lectures in Math., pages 223--239. Univ.
  Chicago Press, Chicago, IL, 1999.

\bibitem{kenig}
Carlos~E. Kenig.
\newblock Some recent applications of unique continuation.
\newblock In {\em Recent developments in nonlinear partial differential
  equations}, volume 439 of {\em Contemp. Math.}, pages 25--56. Amer. Math.
  Soc., Providence, RI, 2007.

\bibitem{KT}
Herbert Koch and Daniel Tataru.
\newblock Carleman estimates and unique continuation for second-order elliptic
  equations with nonsmooth coefficients.
\newblock {\em Comm. Pure Appl. Math.}, 54(3):339--360, 2001.

\bibitem{K}
Igor Kukavica.
\newblock Quantitative uniqueness for second-order elliptic operators.
\newblock {\em Duke Math. J.}, 91(2):225--240, 1998.

\bibitem{linakawangduke}
Ching-Lung Lin, Gen Nakamura, and Jenn-Nan Wang.
\newblock Optimal three-ball inequalities and quantitative uniqueness for the
  {L}am\'e system with {L}ipschitz coefficients.
\newblock {\em Duke Math. J.}, 155(1):189--204, 2010.

\bibitem{Linakawangib}
Ching-Lung Lin, Gen Nakamura, and Jenn-Nan Wang.
\newblock Quantitative uniqueness for second order elliptic operators with
  strongly singular coefficients.
\newblock {\em Rev. Mat. Iberoam.}, 27(2):475--491, 2011.

\bibitem{L}
Fang-Hua Lin.
\newblock Nodal sets of solutions of elliptic and parabolic equations.
\newblock {\em Comm. Pure Appl. Math.}, 44(3):287--308, 1991.

\bibitem{Regs}
Rachid Regbaoui.
\newblock Unique continuation for differential equations of {S}chr\"odinger's
  type.
\newblock {\em Comm. Anal. Geom.}, 7(2):303--323, 1999.

\bibitem{Sogge}
Christopher~D. Sogge.
\newblock Strong uniqueness theorems for second order elliptic differential
  equations.
\newblock {\em Amer. J. Math.}, 112(6):943--984, 1990.

\bibitem{Wolff}
T.~H. Wolff.
\newblock A property of measures in {${\bf R}\sp N$} and an application to
  unique continuation.
\newblock {\em Geom. Funct. Anal.}, 2(2):225--284, 1992.

\end{thebibliography}
\end{document}